\documentclass[a4paper,final,12pt]{article}

\usepackage{amsmath,amssymb,amsthm,amsfonts,
geometry,listings,graphicx,showkeys,bbm,hyperref,
enumerate,nicefrac,mathtools}

\newtheorem{lemma}{Lemma}[section]
\newtheorem{theorem}[lemma]{Theorem}

\newtheorem{corollary}[lemma]{Corollary}

\newcommand{\YMNN}[1]{\ensuremath{\mathcal{X}_{#1}^{M,N}}}

\newcommand{\OMNN}[1]{\ensuremath{\mathcal{O}_{#1}^{M,N}}}

\renewcommand{\S}[0]{\ensuremath{\mathcal{S}}}

\newcommand{\Id}[0]{\ensuremath{\operatorname{Id}}}
\renewcommand{\P}[0]{\ensuremath{\mathbb{P}}}

\newcommand{\R}[0]{\ensuremath{\mathbb{R}}}
\newcommand{\N}[0]{\ensuremath{\mathbb{N}}}
\newcommand{\E}[0]{\ensuremath{\mathbb{E}}}
\newcommand{\B}[0]{\ensuremath{\mathcal{B}}}

\newcommand{\F}[0]{\ensuremath{\mathcal{F}}}
\renewcommand{\L}[0]{\ensuremath{\mathcal{L}}}
\newcommand{\fl}[2][T\!/\!M]{\ensuremath{ \lfloor #2 \rfloor_{#1} }}
\newcommand{\cl}[2][T\!/\!M]{\ensuremath{ \lceil #2 \rceil_{#1} }}

\newcommand{\one}[0]{\ensuremath{\mathbbm{1}}}

\hypersetup{
  linktocpage=true
}

\usepackage{courier}
\title{Lower and upper bounds for\\
strong approximation errors for numerical\\ approximations
of stochastic heat equations}
\author{
Sebastian Becker$^1$,
Benjamin Gess$^{2,3}$, 
Arnulf Jentzen$^{4,5}$,\\
and Peter E. Kloeden$^6$
\bigskip
\\
\small{$^1$ RiskLab, Department of Mathematics, ETH Zurich,}\\
\small{8092 Zurich, Switzerland, e-mail: sebastian.becker@math.ethz.ch}
\smallskip
\\
\small{$^2$ Max Planck Institute for Mathematics in the Sciences,}\\
\small{04103 Leipzig, Germany, e-mail: bgess@mis.mpg.de}
\smallskip
\\
\small{$^3$ Faculty of Mathematics, University of Bielefeld,}\\
\small{33615 Bielefeld, Germany, e-mail: benjamin.gess@gmail.com}
\smallskip
\\
\small{$^4$ Faculty of Mathematics and Computer Science, University of M\"unster,}\\
\small{48149 M\"unster, Germany, e-mail: ajentzen@uni-muenster.de}
\smallskip
\\
\small{$^5$ SAM, Department of Mathematics, ETH Zurich,}\\
\small{8092 Zurich, Switzerland, e-mail: arnulf.jentzen@sam.math.ethz.ch}
\smallskip
\\
\small{$^6$ Mathematics Institute, Goethe University Frankfurt,}\\
\small{60325 Frankfurt am Main, Germany, e-mail: kloeden@math.uni-frankfurt.de}
}

\begin{document}

\maketitle

\begin{abstract}
This article establishes optimal upper and lower error estimates for strong full-discrete numerical 
approximations of the stochastic heat equation driven by space-time white noise. Thereby, this work 
proves the optimality of the strong convergence rates for certain full-discrete approximations of 
stochastic Allen-Cahn equations with space-time white noise which have been obtained in a recent 
previous work of the authors of this article.
\end{abstract}


\section{Introduction}
\label{sec:intro}

In this work we consider space-time discrete numerical methods for linear 
stochastic heat partial differential equations of the type
\begin{equation}
\label{eq:intro_SHE}
    dX_t(x) = \Delta X_t(x) \, dt + dW_t(x)
\end{equation}
with 
zero Dirichlet boundary conditions
$ X_t(0) = X_t(1) = 0 $
for $ t \in [0,T] $, $ x \in (0,1) $
where $ T \in (0,\infty) $ is the time horizon under consideration 
and where $ \frac{ dW }{ dt } $ is a space-time white noise on a probability
space $ (\Omega,\F,\P) $. 
In particular, we analyse strong rates of convergence 
of a full-discrete exponential Euler method, proving optimal upper and lower estimates 
on the strong rate of convergence. 
The next result, Theorem~\ref{thm:intro} below, summarizes 
the main findings of this article.


\begin{theorem}
\label{thm:intro}
Let $ T \in (0,\infty) $, $ p \in [2,\infty) $,
$ 
  ( H, \langle \cdot, \cdot \rangle_H, \left\| \cdot \right\|\!_H ) 
  =
  ( 
    L^2((0,1); \R), 
$
$
    \langle \cdot, \cdot \rangle_{ L^2((0,1); \R) },
$
$
    \left\| \cdot \right\|\!_{ L^2((0,1); \R) }
  )
$,
$ (P_n)_{ n \in \N } \subseteq L(H) $,
let $ (e_n)_{ n \in \N } \subseteq H $
be an orthonormal basis of $H$, 
let $ A \colon D(A) \subseteq H \rightarrow H $
be the Laplacian with zero Dirichlet boundary conditions on $ H $, 
assume for all 
$ n \in \N $, 
$ u \in H $ that
$ A e_n = -\pi^2 n^2 e_n $
and 
$ P_n(u) = \sum_{ k=1 }^n \langle e_k, u \rangle_H \, e_k $,
let $ ( \Omega, \F, \P ) $ be a probability space,
let $ (W_t)_{ t \in [0,T] } $ be an $ \Id_H $-cylindrical 
Wiener process, and
let $ X \colon [0,T] \times \Omega \rightarrow H $
and $ \YMNN{} \colon [0,T] \times \Omega \rightarrow H $,
$ M, N \in \N $, be stochastic processes which satisfy that
for all $ t \in [0,T] $, $ M \in \N $, $N \in \N $
it holds $\P$-a.s.\ that
$
  X_t 
  =
  \int_0^t
  e^{(t-s)A} \, dW_s
$
and 
$
  \YMNN{t} 
  =
  \int_0^t
  P_N
  e^{(t-\max(
    \{ 0, \nicefrac{T}{M}, \nicefrac{2T}{M}, \ldots \} \cap [0, s]
  ))A} \, dW_s
$.
Then there exist $ c, C \in (0,\infty) $
such that
\begin{enumerate}[(i)]
 \item 
  we have for all $ M \in \N $
  that
  \begin{align}
  \label{eq:intro_conv_bounds_1}
  \begin{split}
   &c \,
    M^{-\nicefrac{1}{4}}
    \leq 
    \adjustlimits\lim_{ n \rightarrow \infty }
    \sup_{ t \in [0,T] }
    \left(
      \E\big[ 
        \|
          X_t - \mathcal{X}_t^{M,n}
        \|_{ H }^p 
      \big] 
    \right)^{\!\nicefrac{1}{p}}
    \leq
    C
    M^{-\nicefrac{1}{4}}
  \end{split}
  \end{align}
  and 
 \item 
  we have for all $ N \in \N $
  that
  \begin{align}
  \label{eq:intro_conv_bounds_2}
  \begin{split}
   &c \,
    N^{-\nicefrac{1}{2}}
    \leq 
    \adjustlimits\lim_{ m \rightarrow \infty }
    \sup_{ t \in [0,T] }
    \left(
      \E\big[
        \|
          X_t - \mathcal{X}_t^{m,N}
        \|_{ H }^p
      \big] 
    \right)^{\!\nicefrac{1}{p}}
    \leq
    C 
    N^{-\nicefrac{1}{2}} .
  \end{split}
  \end{align}
\end{enumerate}
\end{theorem}

Theorem~\ref{thm:intro} above is an immediate consequence
of Corollary~\ref{cor:lower_bound_1} below, Lemma~\ref{lem:lower_bound_2}
below, and Da Prato \& Zabczyk~\cite[Lemma~7.7]{dz92}. 
The recent article \cite{BeckerGessJentzenKloeden2017} establishes strong convergence rates 
for suitable space-time discrete approximation methods for 
stochastic Allen-Cahn equations of the type
\begin{equation}
\label{eq:intro_AC}
  dX_t(x) = \Delta X_t(x) \, dt + \big[ \, a \, X_t(x) - b \, \big( X_t(x) \big)^3 \, \big] \, dt + dW_t(x)
\end{equation}
with zero Dirichlet boundary conditions $ X_t(0) = X_t(1) = 0 $ 
for $ t \in [0,T] $, $ x \in (0,1) $ 
where $ a, b \in [0,\infty) $ are real numbers. Roughly speaking, 
in \cite[Theorem~1.1]{BeckerGessJentzenKloeden2017} a spatial convergence rate 
of the order $ \nicefrac{1}{2} - \varepsilon $ and a 
temporal convergence rate of the order $ \nicefrac{1}{4} - \varepsilon $ have been established. 
More precisely, \cite[Theorem~1.1]{BeckerGessJentzenKloeden2017} shows that 
for every $p,\varepsilon \in (0,\infty)$ there exists $C\in\R$ such 
that for all $M,N\in\N$ we have
\begin{equation}
   \sup_{ t \in [0,T] }
   \left(
      \E\big[
        \| X_t - \mathcal{X}_t^{M,N} \|_{ L^2( (0,1) ; \R ) }^p
      \big]
   \right)^{\nicefrac{1}{p}}
   \le 
   C\left(M^{(\varepsilon-\frac{1}{4})}+N^{(\varepsilon-\frac{1}{2})}\right)
\end{equation}
where $ ( \mathcal{X}^{M,N}_t )_{ t \in [0,T] } $ denotes the nonlinearity-truncated approximation scheme
in~\cite{BeckerGessJentzenKloeden2017} applied to~\eqref{eq:intro_AC}. 
The results of this article, that is, inequalities~\eqref{eq:intro_conv_bounds_1} 
and~\eqref{eq:intro_conv_bounds_2}, prove that these rates are essentially (up to an arbitrarily small polynomial order of convergence) 
optimal. We also refer, e.g., to 
\cite{HutzenthalerJentzenKloeden2012,WangGan2013,HutzenthalerJentzenKloeden2013,HutzenthalerJentzen2012,TretyakovZhang2013,Sabanis2013,Sabanis2013E,GoengySabanisS2014,jp2015,JentzenPusnik2016,HutzenthalerJentzenSalimova2016,BeckerJentzen2016,SabanisZhang2017,JentzenSalimovaWelti2017,BrehierGoudenege2018,BrehierCiuHong2018,Wang2018,QiWang2018} for further research articles on 
explicit approximation schemes for stochastic differential
equations with superlinearly growing non-linearities. 
Furthermore, related lower bounds 
for approximation errors in the linear case (i.e., 
in the case $ a = b = 0 $ in \eqref{eq:intro_AC}) can, e.g., be found in
M{\"u}ller-Gronbach, Ritter, \& Wagner~\cite[Theorem~1]{mgritterwager2006},
M{\"u}ller-Gronbach \& Ritter~\cite[Theorem~1]{mgritter2007},
M{\"u}ller-Gronbach, Ritter, \& Wagner~\cite[Theorem~4.2]{mgritterwagner2008},
Conus, Jentzen, \& Kurniawan~\cite[Lemma~6.2]{ConusJentzenKurniawan2014},
and~Jentzen \& Kurniawan~\cite[Corollary~9.4]{JentzenKurniawan2015}.

\subsection{Acknowledgments}
This work has been partially supported
through the SNSF-Research project
200021\_156603
``Numerical approximations of nonlinear
stochastic ordinary and partial differential
equations''.
The third author acknowledges funding by the Deutsche Forschungsgemeinschaft (DFG, German Research Foundation) under Germany’s Excellence Strategy EXC 2044-390685587, Mathematics Muenster: Dynamics-Geometry-Structure.
B. Gess acknowledges financial support by the DFG through the CRC 1283 
``Taming uncertainty and profiting from randomness and low 
regularity in analysis, stochastics and their applications''.

\section{Lower and upper bounds for
strong approximation errors of numerical approximations
of linear stochastic heat equations}
\label{sec:lower_bounds_section}

\subsection{Setting}
\label{sec:lower_bounds_setting}
Let $ \fl[h]{\cdot} \colon \R \rightarrow \R $, 
$ h \in (0,\infty) $,
be the functions which satisfy
for all $ h \in (0,\infty) $, $ t \in \R $ that 
$
  \fl[h]{t}
  =
  \max\big(
    \{ 0, h,$ $ -h, 2h, -2h, \ldots \} \cap (-\infty, t]
  \big)
$.
For every measure space
$ (\Omega, \F, \chi) $,
every measurable space 
$ (S, \S) $,
every set $ R $,
and every
function $ f \colon \Omega \rightarrow R $
let $ [f]_{\chi,\S} $ be
the set given by
$ 
  [f]_{\chi,\S} 
  =
  \{
    g \colon \Omega \rightarrow S
    \colon
    [
    ( 
      \exists \, A \in \F
      \colon
      \chi(A) = 0
      \text{ and }
      \{
        \omega \in \Omega
        \colon
        f(\omega)
        \neq
        g(\omega)
      \}
      \subseteq A
    )
    \cap
    (
      \forall \,
      A \in \S
      \colon 
      g^{-1}(A) \in \F
    )
    ]
  \}
$.
Let $ T, \nu \in (0,\infty) $,
$ 
  ( H, \langle \cdot, \cdot \rangle_H, $ $ \left\| \cdot \right\|\!_H ) 
$
$
  =
  ( 
    L^2(\lambda_{(0,1)}; \R), 
$
$
    \langle \cdot, \cdot \rangle_{ L^2(\lambda_{(0,1)}; \R) },
    \left\| \cdot \right\|\!_{ L^2(\lambda_{(0,1)}; \R) }
  )
$,
$ (e_n)_{ n \in \N } \subseteq H $,
$ (P_n)_{ n \in \N \cup \{ \infty \} } \subseteq L(H) $
satisfy for all $ m \in \N $, $ n \in \N \cup \{ \infty \} $, 
$ u \in H $ that
$ e_m = [ (\sqrt{2}\sin(m \pi x))_{ x \in (0,1) } ]_{ \lambda_{(0,1)}, \B(\R) } $
and 
$ P_n(u) = \sum_{ k=1 }^n \langle e_k, u \rangle_H \, e_k $.
Let $ A \colon D(A) \subseteq H \rightarrow H $
be the Laplacian with zero Dirichlet boundary conditions on $ H $
times the real number $ \nu $.
Let $ ( \Omega, \F, \P ) $ be a probability space.
Let $ (W_t)_{ t \in [0,T] } $ be an $ \Id_H $-cylindrical 
Wiener process. 
Let $ O \colon [0,T] \times \Omega \rightarrow H $
and $ \OMNN{} \colon [0,T] \times \Omega \rightarrow H $,
$ M, N \in \N $, be stochastic processes which satisfy
for all $ t \in [0,T] $, $ M \in \N $, $N \in \N \cup \{ \infty \} $ that
$
  [ O_t ]_{ \P, \B(H) } 
  =
  \int_0^t
  e^{(t-s)A} \, dW_s
$
and 
$
  [ \OMNN{t} ]_{ \P, \B(H) } 
  =
  \int_0^t
  P_N
  e^{(t-\fl{s})A} \, dW_s
$.

\subsection{Lower and upper bounds for
Hilbert-Schmidt norms of Hilbert-Schmidt operators}

\begin{lemma}
\label{lem:eA_monoton}
Assume the setting in Section~\ref{sec:lower_bounds_setting}
and let
$ N \in \N \cup \{ \infty \} $,
$ s_1, s_2, t \in [0,\infty) $ with $ s_1 \leq s_2 $.
Then 
\begin{enumerate}[(i)]
 \item\label{it:eA_monoton_1} we have that
\begin{equation}
  \left( 
    \sum_{ n = 1 }^{\infty}
    \|
      P_N e^{s_1 A}
      ( \Id_H - e^{tA} ) \,
      e_n
    \|_{ H }^2
  \right)^{\!\nicefrac{1}{2}}
  \geq
  \left( 
    \sum_{ n = 1 }^{\infty}
    \|
      P_N e^{s_2 A}
      ( \Id_H - e^{tA} ) \,
      e_n
    \|_{ H }^2
  \right)^{\!\nicefrac{1}{2}}
\end{equation}
and 
\item\label{it:eA_monoton_2} we have that
\begin{equation}
  \left( 
    \sum_{ n = 1 }^{\infty}
    \|
      P_N e^{t A}
      ( \Id_H - e^{s_1 A} ) \,
      e_n
    \|_{ H }^2
  \right)^{\!\nicefrac{1}{2}}
  \leq 
  \left( 
    \sum_{ n = 1 }^{\infty}
    \|
      P_N e^{t A}
      ( \Id_H - e^{s_2 A} ) \,
      e_n
    \|_{ H }^2
  \right)^{\!\nicefrac{1}{2}} .
\end{equation}
\end{enumerate}

\end{lemma}

\begin{proof}[Proof of Lemma~\ref{lem:eA_monoton}]
Throughout this proof let 
$ ( \mu_n )_{ n \in \N } \subseteq \R $
satisfy for all $ n \in \N $ that
$ \mu_n = \nu \pi^2 n^2 $.
Next observe that
\begin{align}
\begin{split}
 &\sum_{ n = 1 }^{\infty}
  \|
    P_N e^{s_1 A}
    ( \Id_H - e^{tA} ) \, 
    e_n
  \|_{ H }^2
\\&=
  \sum_{ n = 1 }^N
  \| 
    e^{s_1 A}
    ( \Id_H - e^{tA} ) \,
    e_n
  \|_H^2
  =
  \sum_{ n = 1 }^N
  \| 
    e^{-\mu_n s_1}
    ( 1 - e^{-\mu_n t} ) \,
    e_n
  \|_H^2
\\&=
  \sum_{ n = 1 }^N
  | 
    e^{-\mu_n s_1}
    ( 1 - e^{-\mu_n t} )
  |^2
  \geq
  \sum_{ n = 1 }^N
  | 
    e^{-\mu_n s_2}
    ( 1 - e^{-\mu_n t} )
  |^2
\\&=
  \sum_{ n = 1 }^N
  \| 
    e^{s_2 A}
    ( \Id_H - e^{tA} ) \,
    e_n
  \|_H^2
  =
  \sum_{ n = 1 }^{\infty}
  \|
    P_N e^{s_2 A}
    ( \Id_H - e^{tA} ) \,
    e_n
  \|_{ H }^2 .
\end{split} 
\end{align}
This establishes~\eqref{it:eA_monoton_1}.
Moreover, note that
\begin{align}
\begin{split}
 &\sum_{ n = 1 }^{\infty}
  \|
    P_N e^{t A}
    ( \Id_H - e^{s_1 A} ) \,
    e_n
  \|_{ H }^2
\\&=
  \sum_{ n = 1 }^N
  \| 
    e^{t A}
    ( \Id_H - e^{s_1 A} ) \,
    e_n
  \|_H^2
  =
  \sum_{ n = 1 }^N
  \| 
    e^{-\mu_n t}
    ( 1 - e^{-\mu_n s_1} ) \,
    e_n
  \|_H^2
\\&=
  \sum_{ n = 1 }^N
  | 
    e^{-\mu_n t}
    ( 1 - e^{-\mu_n s_1} )
  |^2
  \leq
  \sum_{ n = 1 }^N
  | 
    e^{-\mu_n t}
    ( 1 - e^{-\mu_n s_2} )
  |^2
\\&=
  \sum_{ n = 1 }^N
  \| 
    e^{t A}
    ( \Id_H - e^{s_2 A} ) \,
    e_n
  \|_H^2
  =
  \sum_{ n = 1 }^{\infty}
  \|
    P_N e^{t A}
    ( \Id_H - e^{s_2 A} ) \,
    e_n
  \|_{ H }^2 . 
\end{split}
\end{align}
The proof of Lemma~\ref{lem:eA_monoton}
is thus completed.
\end{proof}

\begin{lemma}
\label{lem:eA_bounds}
Assume the setting in Section~\ref{sec:lower_bounds_setting}
and let $ N \in \N \cup \{ \infty \} $, 
$ t \in (0,T] $. Then
\begin{multline}
  \left[ 
    \int_0^{\max\{0, t(N+1)^2-(1+\sqrt{t})^2\}} 
    \frac{ 
      ( 1 - e^{-\nu\pi^2 \min\{1,tN^2\}} )^2
    }{
      2 \nu \pi^2 (x+[1+\sqrt{T}]^2)^{\nicefrac{3}{2}} 
    } \, dx
  \right]^{\nicefrac{1}{2}}
\\
  \leq
  \|
    P_N
    (-\sqrt{t}A)^{-\nicefrac{1}{2}}
    ( \Id_H - e^{tA} )
  \|_{ HS(H) }
  \leq
  \left[ 
    \tfrac{ 
      1
    }{
      \pi \sqrt{\nu}
    }
    +
    \tfrac{ 
      1 
    }{
      \nu \pi^2
    }
    +
    4
    \pi
    \sqrt{\nu}
  \right]^{\nicefrac{1}{2}} .
\end{multline}

\end{lemma}
\begin{proof}[Proof of Lemma~\ref{lem:eA_bounds}]
Observe that
\begin{align}
\begin{split}
 &\tfrac{1}{\sqrt{t}}
  \|
    P_N
    (-A)^{-\nicefrac{1}{2}}
    ( \Id_H - e^{tA} )
  \|_{ HS(H) }^2
\\&=
  \tfrac{1}{\sqrt{t}}
  \sum_{ k=1 }^N
  \| 
    (-A)^{-\nicefrac{1}{2}}
    ( \Id_H - e^{tA} ) e_k
  \|_H^2
\\&=
  \tfrac{1}{\sqrt{t}}
  \sum_{ k=1 }^N
  \| 
    ( \nu \pi^2 k^2 )^{-\nicefrac{1}{2}}
    ( 1 - e^{-\nu\pi^2 k^2 t} ) e_k
  \|_H^2
\\&=
  \sum_{ k=1 }^N
  \frac{ 
    ( 1 - e^{-\nu\pi^2 k^2 t} )^2
  }{
    \nu \pi^2 k^2 \sqrt{t}
  }
  =
  \sum_{ k=1 }^N
  \int_k^{k+1} 
  \frac{ 
    ( 1 - e^{-\nu\pi^2 k^2 t} )^2
  }{
    \nu \pi^2 k^2 \sqrt{t}
  } \, dx
\\&\geq 
  \sum_{ k=1 }^N
  \int_k^{k+1} 
  \frac{ 
    ( 1 - e^{-\nu\pi^2 (x-1)^2 t} )^2
  }{
    \nu \pi^2 x^2 \sqrt{t}
  } \, dx 
\\&=
  \int_1^{N+1} 
  \frac{ 
    ( 1 - e^{-\nu\pi^2 (x-1)^2 t} )^2
  }{
    \nu \pi^2 x^2 \sqrt{t}
  } \, dx 
\\&\geq 
  \int_{1+\min\{\nicefrac{1}{\sqrt{t}}, N\}}^{N+1} 
  \frac{ 
    ( 1 - e^{-\nu\pi^2 (x-1)^2 t} )^2
  }{
    \nu \pi^2 x^2 \sqrt{t}
  } \, dx .
\end{split}
\end{align}
This and the integral transformation theorem imply that
\begin{align}
\label{eq:eA_bounds_1}
\begin{split}
 &\tfrac{1}{\sqrt{t}}
  \|
    P_N
    (-A)^{-\nicefrac{1}{2}}
    ( \Id_H - e^{tA} )
  \|_{ HS(H) }^2
\\&\geq
  \int_{1+\min\{\nicefrac{1}{\sqrt{t}}, N\}}^{N+1} 
  \frac{ 
    ( 1 - e^{-\nu\pi^2 \min\{1,tN^2\}} )^2
  }{
    \nu \pi^2 x^2 \sqrt{t}
  } \, dx
\\&=
  \int_{(1+\min\{\nicefrac{1}{\sqrt{t}}, N\})^2}^{(N+1)^2} 
  \frac{ 
    ( 1 - e^{-\nu\pi^2\min\{1,tN^2\}} )^2
  }{
    2 \nu \pi^2 x \sqrt{xt} 
  } \, dx
\\&=
  \int_{t(1+\min\{\nicefrac{1}{\sqrt{t}}, N\})^2}^{t(N+1)^2}
  \frac{ 
    ( 1 - e^{-\nu\pi^2 \min\{1,tN^2\}} )^2
  }{
    2 \nu \pi^2 x \sqrt{x} 
  } \, dx
\\&=
  \int_{\min\{(1+\sqrt{t})^2,t(N+1)^2\}}^{t(N+1)^2}
  \frac{ 
    ( 1 - e^{-\nu\pi^2 \min\{1,tN^2\}} )^2
  }{
    2 \nu \pi^2 x \sqrt{x} 
  } \, dx
\\&=
  \int_0^{t(N+1)^2-\min\{(1+\sqrt{t})^2,t(N+1)^2\}} 
  \frac{ 
    ( 1 - e^{-\nu\pi^2 \min\{1,tN^2\}} )^2
  }{
    2 \nu \pi^2 (x+\min\{(1+\sqrt{t})^2,t(N+1)^2\})^{\nicefrac{3}{2}} 
  } \, dx 
\\&\geq
  \int_0^{\max\{0, t(N+1)^2-(1+\sqrt{t})^2\}} 
  \frac{ 
    ( 1 - e^{-\nu\pi^2 \min\{1,tN^2\}} )^2
  }{
    2 \nu \pi^2 (x+[1+\sqrt{T}]^2)^{\nicefrac{3}{2}} 
  } \, dx .
\end{split}
\end{align}
Moreover, note that
\begin{align}
\begin{split}
 &\tfrac{1}{\sqrt{t}}
  \|
    P_N
    (-A)^{-\nicefrac{1}{2}}
    ( \Id_H - e^{tA} )
  \|_{ HS(H) }^2
\\&=
  \tfrac{1}{\sqrt{t}}
  \sum_{ k=1 }^N
  \| 
    (-A)^{-\nicefrac{1}{2}}
    ( \Id_H - e^{tA} ) e_k
  \|_H^2
\\&=
  \tfrac{1}{\sqrt{t}}
  \sum_{ k=1 }^N
  \| 
    ( \nu \pi^2 k^2 )^{-\nicefrac{1}{2}}
    ( 1 - e^{-\nu\pi^2 k^2 t} ) e_k
  \|_H^2
\\&=
  \sum_{ k=1 }^N
  \frac{ 
    ( 1 - e^{-\nu\pi^2 k^2 t} )^2
  }{
    \nu \pi^2 k^2 \sqrt{t}
  } 
\\&=
  \frac{ 
    ( 1 - e^{-\nu\pi^2 t} )^2
  }{
    \nu \pi^2 \sqrt{t}
  }
  +
  \sum_{ k=2 }^N
  \int_{k-1}^k
  \frac{ 
    ( 1 - e^{-\nu\pi^2 k^2 t} )^2
  }{
    \nu \pi^2 k^2 \sqrt{t}
  } \, dx .
\end{split}
\end{align}
The fact that 
\begin{equation}
  \forall \, x \in (0,\infty)
  ,
  r \in [0,1]
  \colon
  x^{-r} 
  (1-e^{-x}) 
  \leq 
  1
  ,
\end{equation}
the fact that
\begin{equation}
  \forall \,
  x \in [1,\infty)
  \colon
  (x+1)^2 \leq 4 x^2
  ,
\end{equation}
and the integral transformation theorem
hence yield that
\begin{align}
\begin{split}
 &\tfrac{1}{\sqrt{t}}
  \|
    P_N
    (-A)^{-\nicefrac{1}{2}}
    ( \Id_H - e^{tA} )
  \|_{ HS(H) }^2
\\&\leq
  \frac{ 
    ( 1 - e^{-\nu\pi^2 t} )^{\nicefrac{3}{2}}
  }{
    \pi \sqrt{\nu} 
  }
  +
  \sum_{ k=2 }^N
  \int_{k-1}^k
  \frac{ 
    ( 1 - e^{-\nu\pi^2 (x+1)^2 t} )^2
  }{
    \nu \pi^2 x^2 \sqrt{t}
  } \, dx
\\&\leq 
  \frac{ 
    1
  }{
    \pi \sqrt{\nu}
  }
  +
  \int_{1}^{N}
  \frac{ 
    ( 1 - e^{-4\nu\pi^2 x^2 t} )^2
  }{
    \nu \pi^2 x^2 \sqrt{t}
  } \, dx
\\&=
  \frac{ 
    1
  }{
    \pi \sqrt{\nu}
  }
  +
  \int_{1}^{N^2}
  \frac{ 
    ( 1 - e^{-4\nu\pi^2 x t} )^2
  }{
    2 \nu \pi^2 x \sqrt{x t}
  } \, dx
\\&=
  \frac{ 
    1
  }{
    \pi \sqrt{\nu}
  }
  +
  \int_{t}^{tN^2}
  \frac{ 
    ( 1 - e^{-4\nu\pi^2 x} )^2
  }{
    2 \nu \pi^2 x \sqrt{x}
  } \, dx .
\end{split}
\end{align}
Again the fact that 
\begin{equation}
  \forall \, x \in (0,\infty)
  ,
  r \in [0,1]
  \colon
  x^{-r} 
  (1-e^{-x}) 
  \leq 
  1
\end{equation}
therefore ensures that
\begin{align}
\begin{split}
 &\tfrac{1}{\sqrt{t}}
  \|
    P_N
    (-A)^{-\nicefrac{1}{2}}
    ( \Id_H - e^{tA} )
  \|_{ HS(H) }^2
\\&\leq 
  \frac{ 
    1
  }{
    \pi \sqrt{\nu}
  }
  +
  \int_{0}^{\infty}
  \frac{ 
    ( 1 - e^{-4\nu\pi^2 x} )^2
  }{
    2 \nu \pi^2 x \sqrt{x}
  } \, dx
\\  & \leq
  \frac{ 
    1
  }{
    \pi \sqrt{\nu}
  }
  +
  2
  \int_{0}^{1}
  \frac{ 
    ( 1 - e^{-4\nu\pi^2 x} )
  }{
    \sqrt{x}
  } \, dx
  +
  \int_{1}^{\infty}
  \frac{ 
    1
  }{
    2 \nu \pi^2 x \sqrt{x}
  } \, dx
\\&\leq
  \frac{ 
    1
  }{
    \pi \sqrt{\nu}
  }
  +
  4
  \pi
  \sqrt{\nu}
  \int_{0}^{1}
  \sqrt{ 1 - e^{-4\nu\pi^2 x} } \, dx
  +
  \left[ 
    \frac{ 
      -1 
    }{
      \nu \pi^2 \sqrt{x}
    }
  \right]_{ x=1 }^{ x = \infty }
\\ &
\leq
  \frac{ 
    1
  }{
    \pi \sqrt{\nu}
  }
  +
  4
  \pi
  \sqrt{\nu}
  +
  \frac{ 
    1 
  }{
    \nu \pi^2
  } .
\end{split}
\end{align}
Combining this and~\eqref{eq:eA_bounds_1}
completes the proof of Lemma~\ref{lem:eA_bounds}.
\end{proof}

\subsection{Lower and upper bounds for 
strong approximation errors of temporal discretizations
of linear stochastic heat equations}

\begin{lemma}
\label{lem:lower_bound_1}
Assume the setting in Section~\ref{sec:lower_bounds_setting}
and let $ M \in \N $, $ N \in \N \cup \{ \infty \} $. Then
\begin{align}
\begin{split}
 &\frac{1}{M^{\nicefrac{1}{4}}}
  \left[ 
    \int_0^{\max\left\{0, \frac{T(N+1)^2}{2M}-\left[1+\frac{\sqrt{T}}{\sqrt{2M}}\right]^{2}\right\}} 
    \frac{ 
      \sqrt{T}
      \left[ 1 - e^{-\nu \pi^2 T} \right] \!
      \left[ 1 - \exp(-\nu\pi^2 \min\{1,\frac{TN^2}{2M}\}) \right]^2
    }{
      8 \nu \pi^2 \sqrt{2} 
      (x+[1+\sqrt{T}]^2)^{\nicefrac{3}{2}} 
    } \, dx
  \right]^{\nicefrac{1}{2}}
\\&\leq 
 \|
   P_N O_T - \OMNN{T} 
 \|_{ \L^2(\P; H ) }
 \leq
 \sup_{ t \in [0,T] }
 \| 
   P_N O_t - \OMNN{t} 
 \|_{ \L^2(\P; H ) }
\\&=
 \sup_{ t \in [0,T] }
 \left[
   \int_0^t
   \|
     P_N
     e^{(t-s)A}
     ( 
       \Id_H - e^{(s-\fl{s})A} 
     )
   \|_{ HS(H) }^2 \, ds
 \right]^{\nicefrac{1}{2}}
\\&\leq
 \frac{1}{M^{\nicefrac{1}{4}}}
 \left[
    \frac{\sqrt{T}}{ 2 }
    \left( 
      \frac{ 
	1
      }{
	\pi \sqrt{\nu}
      }
      +
      \frac{ 
	1 
      }{
	\nu \pi^2
      }
      +
      4
      \pi
      \sqrt{\nu}
    \right)
  \right]^{\nicefrac{1}{2}} .
\end{split}
\end{align}

\end{lemma}
\begin{proof}[Proof of Lemma~\ref{lem:lower_bound_1}]
Throughout this proof let
$ (\mu_n)_{ n \in \N } \subseteq \R $ satisfy
for all $ n \in \N $ that 
$ \mu_n = \nu \pi^2 n^2 $
and let $ \cl[h]{\cdot} \colon \R \rightarrow \R $, 
$ h \in (0,\infty) $,
be the functions which satisfy
for all $ h \in (0,\infty) $, $ t \in \R $ that 
$
  \cl[h]{t}
  =
  \min\!\left(
    \{ 0, h, -h, 2h, -2h, \ldots \} \cap [t,\infty)
  \right)
$.
Observe that Lemma~\ref{lem:eA_monoton}~\eqref{it:eA_monoton_1} 
ensures for all 
$ t \in [0,T) $ that
\begin{align}
\begin{split}
 &2
  \int_{\fl{t}}^{\fl{t}+\frac{T}{M}}
  \one_{[\fl{s}, \fl{s}+\frac{T}{2M}]}^{\R}(s) \,
  \|
    P_N
    e^{sA}
    ( 
      \Id_H - e^{\frac{T}{2M}A} 
    )
  \|_{ HS(H) }^2 \, ds
\\&=
  2
  \int_{\fl{t}}^{\fl{t}+\frac{T}{2M}}
  \|
    P_N
    e^{sA}
    ( 
      \Id_H - e^{\frac{T}{2M}A} 
    )
  \|_{ HS(H) }^2 \, ds
\\&\geq 
  \int_{\fl{t}}^{\fl{t}+\frac{T}{2M}}
  \|
    P_N
    e^{sA}
    ( 
      \Id_H - e^{\frac{T}{2M}A} 
    )
  \|_{ HS(H) }^2 \, ds
\\&\quad+
  \int_{\fl{t}+\frac{T}{2M}}^{\fl{t}+\frac{T}{M}}
  \|
    P_N
    e^{sA}
    ( 
      \Id_H - e^{\frac{T}{2M}A} 
    )
  \|_{ HS(H) }^2 \, ds
\\&=
  \int_{\fl{t}}^{\fl{t}+\frac{T}{M}}
  \|
    P_N
    e^{sA}
    ( 
      \Id_H - e^{\frac{T}{2M}A} 
    )
  \|_{ HS(H) }^2 \, ds .
\end{split}
\end{align}
Therefore, we obtain that
\begin{align}
\label{eq:lower_bound_1_pre}
\begin{split}
 &2
  \int_{0}^{T}
  \one_{[\fl{s}, \fl{s}+\frac{T}{2M}]}^{\R}(s) \,
  \|
    P_N
    e^{sA}
    ( 
      \Id_H - e^{\frac{T}{2M}A} 
    )
  \|_{ HS(H) }^2 \, ds
\\&\geq 
  \int_{0}^{T}
  \|
    P_N
    e^{sA}
    ( 
      \Id_H - e^{\frac{T}{2M}A} 
    )
  \|_{ HS(H) }^2 \, ds .
\end{split}
\end{align}
Next note that It{\^o}'s isometry implies for all 
$ t \in [0,T] $ that
\begin{align}
\label{eq:lower_bound_1}
\begin{split}
 &\| 
    P_N O_t - \OMNN{t} 
  \|_{ \L^2(\P; H ) }^2
\\&=
  \E\!\left[ 
    \|
      P_N O_t - \OMNN{t}
    \|_H^2
  \right]
  =
  \E\!\left[ 
    \left\|
      \int_0^t
      P_N
      e^{(t-s)A}
      ( 
        \Id_H - e^{(s-\fl{s})A} 
      ) \, dW_s
    \right\|_H^2
  \right]
\\&=
  \int_0^t
  \|
    P_N
    e^{(t-s)A}
    ( 
      \Id_H - e^{(s-\fl{s})A} 
    )
  \|_{ HS(H) }^2 \, ds .
\end{split}
\end{align}
This, the fact that
$ 
  \forall \, s \in [0,T]
  \colon
  T - \fl{T-s} = \cl{s} 
$,
and Lemma~\ref{lem:eA_monoton}~\eqref{it:eA_monoton_2}
ensure that
\begin{align}
\begin{split}
 &\| 
    P_N O_T - \OMNN{T} 
  \|_{ \L^2(\P; H ) }^2
\\&=
  \int_0^T
  \|
    P_N
    e^{sA}
    ( 
      \Id_H - e^{(T-s-\fl{T-s})A} 
    )
  \|_{ HS(H) }^2 \, ds
\\&=
  \int_0^T
  \|
    P_N
    e^{sA}
    ( 
      \Id_H - e^{(\cl{s}-s)A} 
    )
  \|_{ HS(H) }^2 \, ds
\\&\geq
  \int_0^T
  \one_{[\fl{s}, \fl{s}+\frac{T}{2M}]}^{\R}(s) \,
  \|
    P_N
    e^{sA}
    ( 
      \Id_H - e^{(\cl{s}-s)A} 
    )
  \|_{ HS(H) }^2 \, ds
\\&\geq
  \int_0^T
  \one_{[\fl{s}, \fl{s}+\frac{T}{2M}]}^{\R}(s) \,
  \|
    P_N
    e^{sA}
    ( 
      \Id_H - e^{\frac{T}{2M}A} 
    )
  \|_{ HS(H) }^2 \, ds .
\end{split}
\end{align}
Inequality~\eqref{eq:lower_bound_1_pre} 
hence proves that
\begin{align}
\begin{split}
 &\| 
    P_N O_T - \OMNN{T} 
  \|_{ \L^2(\P; H ) }^2
\\&\geq
  \frac{1}{2}
  \left[ 
    2
    \int_0^T
    \one_{[\fl{s}, \fl{s}+\frac{T}{2M}]}^{\R}(s) \,
    \|
      P_N
      e^{sA}
      ( 
	\Id_H - e^{\frac{T}{2M}A} 
      )
    \|_{ HS(H) }^2 \, ds
  \right] 
\\&\geq
  \frac{1}{2}
  \int_0^T
  \|
    P_N
    e^{sA}
    ( 
      \Id_H - e^{\frac{T}{2M}A} 
    )
  \|_{ HS(H) }^2 \, ds 
\\&=
  \frac{1}{2}
  \int_0^T
  \sum_{ k=1 }^N
  \|
    e^{sA}
    ( 
      \Id_H - e^{\frac{T}{2M}A} 
    ) e_k
  \|_H^2 \, ds
\\&=
  \frac{1}{2}
  \int_0^T
  \sum_{ k=1 }^N
  |
    e^{-\mu_k s}
    ( 
      1 - e^{-\mu_k\frac{T}{2M}} 
    )
  |^2 \, ds
\\&=
  \frac{1}{2}
  \sum_{ k=1 }^N
  \frac{ 
    ( 1 - e^{-2\mu_k T} ) 
  }{ 2 \mu_k } \,
  | 1 - e^{-\mu_k\frac{T}{2M}} |^2 .
\end{split}
\end{align}
Lemma~\ref{lem:eA_bounds} 
therefore implies that
\begin{align}
\label{eq:lower_bound_2}
\begin{split}
 &\| 
    P_N O_T - \OMNN{T} 
  \|_{ \L^2(\P; H ) }^2
\\&\geq
  \frac{1}{4}
  ( 1 - e^{-\mu_1 T} ) 
  \sum_{ k=1 }^N
  \left| 
    \frac{ 
      (1 - e^{-\mu_k\frac{T}{2M}})
    }{ \sqrt{\mu_k} }
  \right|^2
\\&=
  \frac{1}{4}
  ( 1 - e^{-\mu_1 T} ) 
  \sum_{ k=1 }^N
  \| 
    (-A)^{-\nicefrac{1}{2}}
    ( 
      \Id_H - e^{\frac{T}{2M}A} 
    ) e_k
  \|_H^2
\\&=
  \frac{1}{4}
  ( 1 - e^{-\mu_1 T} ) 
  \| 
    P_N
    (-A)^{-\nicefrac{1}{2}}
    ( 
      \Id_H - e^{\frac{T}{2M}A} 
    ) 
  \|_{ HS(H) }^2
\\&\geq
  \frac{
    \sqrt{ T }
    ( 1 - e^{-\mu_1 T} ) 
  }{4 \sqrt{2M} }
\\&\quad\cdot
  \left[ 
    \int_0^{\max\left\{0, \frac{T(N+1)^2}{2M}-\left[1+\frac{\sqrt{T}}{\sqrt{2M}}\right]^2\right\}} 
    \frac{ 
      \left[ 1 - \exp(-\nu\pi^2 \min\{1,\frac{TN^2}{2M}\}) \right]^2
    }{
      2 \nu \pi^2 (x+[1+\sqrt{T}]^2)^{\nicefrac{3}{2}} 
    } \, dx
  \right] .
\end{split}
\end{align}
In the next step observe
that~\eqref{eq:lower_bound_1} 
and Lemma~\ref{lem:eA_monoton}~\eqref{it:eA_monoton_2}
assure that
\begin{align}
\begin{split}
\sup_{ t \in [0,T] }
  \| 
    P_N O_t - \OMNN{t} 
  \|_{ \L^2(\P; H ) }^2
&\leq
  \sup_{ t \in [0,T] }
  \int_0^t
  \|
    P_N
    e^{(t-s)A}
    ( \Id_H - e^{\frac{T}{M}A} )
  \|_{ HS(H) }^2 \, ds
\\ & =
  \sup_{ t \in [0,T] }
  \int_0^t
  \|
    P_N
    e^{sA}
    ( \Id_H - e^{\frac{T}{M}A} )
  \|_{ HS(H) }^2 \, ds
\\&=
  \int_0^T
  \|
    P_N
    e^{sA}
    ( \Id_H - e^{\frac{T}{M}A} )
  \|_{ HS(H) }^2 \, ds
\\ & 
=
  \int_0^T
  \sum_{ k=1 }^N
  \|
    e^{sA}
    ( \Id_H - e^{\frac{T}{M}A} )
    e_k
  \|_H^2 \, ds .
\end{split}
\end{align}
Lemma~\ref{lem:eA_bounds}
hence yields that	
\begin{align}
\begin{split}
 &\sup_{ t \in [0,T] }
  \| 
    P_N O_t - \OMNN{t} 
  \|_{ \L^2(\P; H ) }^2
\\&\leq
  \int_0^T
  \sum_{ k=1 }^N
  | 
    e^{-\mu_k s}
    (
      1 - e^{-\mu_k \frac{T}{M} }
    )
  |^2 \, ds
  =
  \sum_{ k=1 }^N
  \frac{ (1-e^{-2\mu_k T} ) }{ 2 \mu_k }
  | 1 - e^{-\mu_k \frac{T}{M} } |^2
\\ &
\leq
  \frac{1}{2}
  \sum_{ k=1 }^N
  \left|
    \frac{ 
      ( 1 - e^{-\mu_k \frac{T}{M} } )
    }{ \sqrt{ \mu_k } }
  \right|^2
  =
  \frac{1}{2}
  \sum_{ k=1 }^N
  \|
    (-A)^{-\nicefrac{1}{2} } 
    ( \Id_H - e^{\frac{T}{M}A } ) e_k
  \|_H^2
\\&
  =
  \frac{1}{2}
  \|
    P_N
    (-A)^{-\nicefrac{1}{2} } 
    ( \Id_H - e^{\frac{T}{M}A } )
  \|_{ HS(H) }^2
\leq
  \frac{\sqrt{T}}{ 2 \sqrt{M} }
  \left[ 
    \frac{ 
      1
    }{
      \pi \sqrt{\nu}
    }
    +
    \frac{ 
      1 
    }{
      \nu \pi^2
    }
    +
    4
    \pi
    \sqrt{\nu}
  \right] .
\end{split}
\end{align}
Combining this with~\eqref{eq:lower_bound_1} and~\eqref{eq:lower_bound_2}
completes the proof of Lemma~\ref{lem:lower_bound_1}.
\end{proof}

In the next result, Corollary~\ref{cor:lower_bound_1},
we specialize Lemma~\ref{lem:lower_bound_1} to the case $ N = \infty $
where no spatial discretization is applied to the stochastic process 
$
  O \colon [0,T] \times \Omega \rightarrow H 
$.

\begin{corollary}
\label{cor:lower_bound_1}
Assume the setting in Section~\ref{sec:lower_bounds_setting}
and let $ M \in \N $. Then
\begin{align}
\begin{split}
 &\frac{1}{M^{\nicefrac{1}{4}}}
  \left[ 
    \int_0^{\infty} 
    \frac{ 
      \sqrt{T}
      ( 1 - e^{-\nu \pi^2 T} ) 
      ( 1 - e^{-\nu\pi^2} )^2
    }{
      8 \nu \pi^2 \sqrt{2} (x+[1+\sqrt{T}]^2)^{\nicefrac{3}{2}} 
    } \, dx
  \right]^{\nicefrac{1}{2}}
\\&\leq
  \liminf_{ N \rightarrow \infty }
  \| 
    P_N O_T - \mathcal{O}_T^{M,N} 
  \|_{ \L^2(\P; H ) }
  =
  \limsup_{ N \rightarrow \infty }
  \| 
    P_N O_T - \mathcal{O}_T^{M,N} 
  \|_{ \L^2(\P; H ) }
\\&=
  \| 
    O_T - \mathcal{O}_T^{M,\infty} 
  \|_{ \L^2(\P; H ) }
  \leq
  \adjustlimits\liminf_{ N \rightarrow \infty }
  \sup_{ t \in [0,T] }
  \| 
    P_N O_t - \mathcal{O}_t^{M,N} 
  \|_{ \L^2(\P; H ) }
\\&=
  \adjustlimits\limsup_{ N \rightarrow \infty }
  \sup_{ t \in [0,T] }
  \| 
    P_N O_t - \mathcal{O}_t^{M,N} 
  \|_{ \L^2(\P; H ) }
  =
  \sup_{ t \in [0,T] }
  \| 
    O_t - \mathcal{O}_t^{M,\infty} 
  \|_{ \L^2(\P; H ) }
\\&\leq
  \frac{1}{M^{\nicefrac{1}{4}}}
  \left[
    \frac{\sqrt{T}}{ 2 }
    \left( 
      \frac{ 
	1
      }{
	\pi \sqrt{\nu}
      }
      +
      \frac{ 
	1 
      }{
	\nu \pi^2
      }
      +
      4
      \pi
      \sqrt{\nu}
    \right)
  \right]^{\nicefrac{1}{2}} .
\end{split}
\end{align}

\end{corollary}

\subsection{Lower and upper bounds for 
strong approximation errors of spatial discretizations
of linear stochastic heat equations}

\begin{lemma}
\label{lem:PN_eA_convergence}
Assume the setting 
in Section~\ref{sec:lower_bounds_setting}.
Then
\begin{align}
 &\limsup_{ M \rightarrow \infty }
  \sup_{ N \in \N }
  \sup_{ t \in [0,T] }
  \left\|
    \int_0^t
    \left(
      P_N
      e^{(t-s)A}
      -
      P_N
      e^{(t-\fl{s})A}
    \right) dW_s
  \right\|_{ L^2(\P; H) }
  =
  0 .
\end{align}
\end{lemma}
\begin{proof}[Proof of Lemma~\ref{lem:PN_eA_convergence}]
Throughout this proof 
let $ \alpha \in (0, \nicefrac{1}{4}) $
and let
$ \beta \in (\nicefrac{1}{4}, \nicefrac{1}{2}-\alpha) $.
Note that
the fact that
$
  4\beta > 1
$
shows that
\begin{align}
\label{eq:PN_eA_convergence_1}
\begin{split}
&
  \sup_{ N \in \N }
  \|
    P_N
  \|_{ HS(H, H_{-\beta}) }^2
\\ &=
  \sup_{ N \in \N }
  \left[ 
    \sum_{ k=1 }^{ N }
    \|
      e_k
    \|_{ H_{-\beta} }^2
  \right]
  =
  \sum_{ k=1 }^{ \infty }
  \|
    (-A)^{-\beta}
    e_k
  \|_{ H }^2
\\&
  =
  \sum_{ k=1 }^{ \infty }
  |
    (\nu \pi^2 k^2)^{-\beta}
  |^2
=
  \sum_{ k=1 }^{ \infty }
  \frac{ 1 }{ ( \sqrt{\nu} \pi k )^{4\beta} }
  <
  \infty .
\end{split}
\end{align}
Next observe that
for all $ M, N \in \N $,
$ t \in [0,T] $
we have that
\begin{align}
\begin{split}
 &\int_0^t
  \|
    P_N
    e^{(t-s)A}
    (
      \Id_H
      -
      e^{(s-\fl{s})A}
    )
  \|_{ HS(H) }^2 \, ds
\\&\leq
  \|
    (-A)^{-\beta}
    P_N
  \|_{ HS(H) }^2
  \int_0^t
  \|
    (-A)^{\beta}
    e^{(t-s)A}
    (
      \Id_H
      -
      e^{(s-\fl{s})A}
    )
  \|_{ L(H) }^2 \, ds
\\&=
  \|
    P_N
  \|_{ HS(H, H_{-\beta}) }^2
  \int_0^t
  \|
    (-A)^{(\alpha+\beta)}
    e^{(t-s)A}
    (-A)^{-\alpha}
    (
      \Id_H
      -
      e^{(s-\fl{s})A}
    )
  \|_{ L( H ) }^2 \, ds
\\&\leq
  \|
    P_N
  \|_{ HS(H, H_{-\beta}) }^2
\\&\quad\cdot
  \int_0^t
  \|
    (-A)^{(\alpha+\beta)}
    e^{(t-s)A}
  \|_{ L(H) }^2
  \|
    (-A)^{-\alpha}
    (
      \Id_H
      -
      e^{(s-\fl{s})A}
    )
  \|_{ L( H ) }^2 \, ds .
\end{split}
\end{align}
The fact that 
\begin{equation}
  \forall \,
  s \in [0,\infty) 
  ,
  r \in [0,1]
  \colon 
  \| (-sA)^r e^{sA} \|_{ L(H) }
  \leq 
  1
\end{equation}
and the fact that
\begin{equation}
  \forall \,
  s \in (0,\infty) 
  ,
  r \in [0,1]
  \colon 
  \| (-sA)^{-r} (\Id_H - e^{sA} ) \|_{ L(H) }
  \leq 
  1
\end{equation}
hence prove for all  
$ M, N \in \N $,
$ t \in [0,T] $ that
\begin{align}
\label{eq:PN_eA_convergence_2}
\begin{split}
 &\int_0^t
  \|
    P_N
    e^{(t-s)A}
    (
      \Id_H
      -
      e^{(s-\fl{s})A}
    )
  \|_{ HS(H) }^2 \, ds
\\&\leq
  \|
    P_N
  \|_{ HS(H, H_{-\beta}) }^2
  \int_0^t
  (t-s)^{-2(\alpha+\beta)}
  (s-\fl{s})^{2\alpha} \, ds
\\&\leq
  \frac{ T^{2\alpha} }{ M^{2\alpha} }
  \|
    P_N
  \|_{ HS(H, H_{-\beta}) }^2
  \int_0^t
  (t-s)^{-2(\alpha+\beta)} \, ds
\\ &
  =
  \frac{ t^{(1-2\alpha-2\beta)} T^{2\alpha}  }{ (1-2\alpha-2\beta) M^{2\alpha} }
  \|
    P_N
  \|_{ HS(H, H_{-\beta}) }^2 .
\end{split}
\end{align}
It{\^o}'s isometry
therefore
ensures for all $ M \in \N $
that
\begin{align}
\begin{split}
 &\sup_{ N \in \N }
  \sup_{ t \in [0,T] }
  \left\|
    \int_0^t
    \left(
      P_N
      e^{(t-s)A}
      -
      P_N
      e^{(t-\fl{s})A}
    \right) dW_s
  \right\|_{ L^2(\P; H) }^2
\\&=
  \sup_{ N \in \N }
  \sup_{ t \in [0,T] }
  \int_0^t
  \|
    P_N
    e^{(t-s)A}
    -
    P_N
    e^{(t-\fl{s})A}
  \|_{ HS(H) }^2 \, ds
\\&\leq
  \frac{ T^{(1-2\beta)} }{ (1-2\alpha-2\beta) M^{2\alpha} }
  \left[
    \sup_{ N \in \N }
    \|
      P_N
    \|_{ HS(H, H_{-\beta}) }^2
  \right] .
\end{split}
\end{align}
Combining this with~\eqref{eq:PN_eA_convergence_1}
completes the 
proof of Lemma~\ref{lem:PN_eA_convergence}.
\end{proof}

\begin{lemma}
\label{lem:lower_bound_2}
Assume the setting in Section~\ref{sec:lower_bounds_setting}
and let $ N \in \N $. Then
\begin{align}
\begin{split}
 &\left[
    \frac{ \sqrt{ 1 - e^{-2 \nu \pi^2 T} } }{ 2 \pi \sqrt{\nu} }  
  \right]
  \frac{ 1 }{ \sqrt{ N } }
  \leq
  \liminf_{ M \rightarrow \infty }
  \| O_T - \OMNN{T} \|_{ \L^2(\P; H) }
\\&=
  \limsup_{ M \rightarrow \infty }
  \| O_T - \OMNN{T} \|_{ \L^2(\P; H) }
  =
  \adjustlimits\liminf_{ M \rightarrow \infty }
  \sup_{ t \in [0,T] }
  \| O_t - \OMNN{t} \|_{ \L^2(\P; H) }
\\&=
  \adjustlimits\limsup_{ M \rightarrow \infty }
  \sup_{ t \in [0,T] }
  \| O_t - \OMNN{t} \|_{ \L^2(\P; H) }
  =
  \|
    O_T
    -
    P_N
    O_T
  \|_{ \L^2(\P; H) }
\\&=
  \sup_{ t \in [0,T] }
  \|
    O_t - P_N O_t
  \|_{ \L^2(\P; H) }
  \leq
  \left[
    \frac{ 1 }{ \pi \sqrt{ 2 \nu } }  
  \right]
  \frac{ 1 }{ \sqrt{ N } } .
\end{split}
\end{align}

\end{lemma}
\begin{proof}[Proof of Lemma~\ref{lem:lower_bound_2}]
Throughout this proof 
let $ (\mu_n)_{ n\in\N } \subseteq \R $
satisfy for all $ n \in \N $ that
\begin{equation}
  \mu_n = \nu \pi^2 n^2 
  .
\end{equation}
Note that Parseval's identity
shows that
for all $ t \in [0,T] $ we have that
\begin{align}
\begin{split}
 &\|
    O_t - P_N O_t
  \|_{ \L^2(\P; H) }^2
\\&=
  \E\!\left[
    \| 
      O_t 
      -
      P_N O_t
    \|_H^2
  \right]
  =
  \E\!\left[
    \sum_{ k = N+1 }^{\infty}
    |
      \langle e_k, O_t \rangle_H
    |^2
  \right]
  =
  \sum_{ k = N+1 }^{\infty}
  \E\!\left[
    |
      \langle e_k, O_t \rangle_H
    |^2
  \right]
\\&=
  \sum_{ k = N+1 }^{\infty}
  \E\!\left[
    \left|
      \int_0^t
      \langle e_k, e^{(t-s)A} dW_s \rangle_H
    \right|^2
  \right]
  =
  \sum_{ k = N+1 }^{\infty}
  \E\!\left[
    \left|
      \int_0^t
      \langle e^{(t-s)A} e_k, dW_s \rangle_H
    \right|^2
  \right] .
\end{split}
\end{align}
It{\^o}'s isometry hence proves for all 
$ t \in [0,T] $ that
\begin{align}
\label{eq:lower_bound_3}
\begin{split}
&
  \|
    O_t - P_N O_t
  \|_{ \L^2(\P; H) }^2
\\
&=
  \sum_{ k = N+1 }^{\infty}
  \E\!\left[
    \left|
      \int_0^t
      e^{-\mu_k(t-s)}
      \, \langle e_k, dW_s \rangle_H
    \right|^2
  \right]
\\&=
  \sum_{ k = N+1 }^{\infty}
  \int_0^t
  e^{-2\mu_k(t-s)} \, ds
  =
  \sum_{ k = N+1 }^{\infty}
  \int_0^t
  e^{-2\mu_k s} \, ds
  =
  \sum_{ k = N+1 }^{\infty}
  \frac{ \left( 1 - e^{-2\mu_k t} \right) }{ 2 \mu_k } .
\end{split}
\end{align}
This shows that
\begin{align}
\label{eq:lower_bound_4}
\begin{split}
 &\sup_{ t \in [0,T] }
  \|
    O_t - P_N O_t
  \|_{ \L^2(\P; H) }^2
  =
  \|
    O_T - P_N O_T
  \|_{ \L^2(\P; H) }^2
\\&=
  \sum_{ k = N+1 }^{\infty}
  \frac{ \left( 1 - e^{-2\mu_k T} \right) }{ 2 \mu_k }
  =
  \sum_{ k = N+1 }^{\infty}
  \frac{ \big( 1 - e^{-2 \nu \pi^2 k^2 T} \big) }{ 2 \nu \pi^2 k^2 }
\\&\geq 
  \left[
    \frac{ 1 - e^{-2 \nu \pi^2 T} }{ 2 \nu \pi^2 }  
  \right]
  \left[ 
    \sum_{ k = N+1 }^{\infty}
    \frac{ 1 }{ k^2 }
  \right]
  \geq
  \left[
    \frac{ 1 - e^{-2 \nu \pi^2 T} }{ 2 \nu \pi^2 }  
  \right]
  \left[ 
    \sum_{ k = N+1 }^{\infty}
    \int_k^{k+1}
    \frac{1}{x^2} \, dx
  \right]
\\&=
  \left[
    \frac{ 1 - e^{-2 \nu \pi^2 T} }{ 2 \nu \pi^2 }  
  \right]
  \left[
    \int_{N+1}^{\infty}
    \frac{1}{x^2} \, dx
  \right]
  =
  \left[
    \frac{ 1 - e^{-2 \nu \pi^2 T} }{ 2 \nu \pi^2 }  
  \right]
  \left[
    - \frac{1}{x}
  \right]_{ x = N+1 }^{ x = \infty } 
\\&=
  \left[
    \frac{ 1 - e^{-2 \nu \pi^2 T} }{ 2 \nu \pi^2 }  
  \right]
  \frac{ 1 }{ ( N + 1 ) }
  \geq
  \left[
    \frac{ 1 - e^{-2 \nu \pi^2 T} }{ 2 \nu \pi^2 }  
  \right]
  \frac{ 1 }{ ( N + N ) }
  =
  \left[
    \frac{ 1 - e^{-2 \nu \pi^2 T} }{ 4 \nu \pi^2 }  
  \right]
  \frac{ 1 }{ N } .
\end{split}
\end{align}
This
implies that
\begin{align}
\label{eq:lower_bound_5}
\begin{split}
 &\sup_{ t \in [0,T] }
  \|
    O_t - P_N O_t
  \|_{ \L^2(\P; H) }^2
\\&=
  \sum_{ k = N+1 }^{\infty}
  \frac{ \big( 1 - e^{-2 \nu \pi^2 k^2 T} \big) }{ 2 \nu \pi^2 k^2 }
  \leq 
  \left[
    \frac{ 1 }{ 2 \nu \pi^2 }  
  \right]
  \left[ 
    \sum_{ k = N+1 }^{\infty}
    \frac{ 1 }{ k^2 }
  \right]
\\&\leq
  \left[
    \frac{ 1 }{ 2 \nu \pi^2 }  
  \right]
  \left[
    \sum_{ k = N+1 }^{\infty}
    \int_{ k-1 }^k
    \frac{ 1 }{ x^2 } \, dx
  \right]
  =
  \left[
    \frac{ 1 }{ 2 \nu \pi^2 }  
  \right]
  \left[ 
    \int_{ N }^{\infty}
    \frac{ 1 }{ x^2 } \, dx
  \right]
\\&=
  \left[
    \frac{ 1 }{ 2 \nu \pi^2 }  
  \right]
  \left[
    - \frac{1}{x}
  \right]_{ x = N }^{ x = \infty }
  =
  \left[
    \frac{ 1 }{ 2 \nu \pi^2 }  
  \right]
  \frac{ 1 }{ N } .
\end{split}
\end{align}
In addition, note that
the triangle inequality
and 
Lemma~\ref{lem:PN_eA_convergence}
prove that
\begin{align}
\label{eq:lower_bound_6}
\begin{split}
 &\adjustlimits\limsup_{ M \rightarrow \infty }
  \sup_{ t \in [0,T] }
  \| O_t - \OMNN{t} \|_{ \L^2(\P; H) }
\\&=
  \adjustlimits\limsup_{ M \rightarrow \infty }
  \sup_{ t \in [0,T] }
  \left\|
    \int_0^t
    ( e^{(t-s)A} - P_N e^{(t-\fl{s})A} ) \, dW_s
  \right\|_{ L^2(\P; H) }
\\&=
  \adjustlimits\limsup_{ M \rightarrow \infty }
  \sup_{ t \in [0,T] }
  \bigg\|\!
    \int_0^t
    ( e^{(t-s)A} - P_N e^{(t-s)A} ) \, dW_s
\\&\quad+
    \int_0^t
    ( P_N e^{(t-s)A} - P_N e^{(t-\fl{s})A} ) \, dW_s
  \bigg\|_{ L^2(\P; H) }
\\&\leq
  \adjustlimits\limsup_{ M \rightarrow \infty }
  \sup_{ t \in [0,T] }
  \left\|
    \int_0^t
    ( e^{(t-s)A} - P_N e^{(t-s)A} ) \, dW_s
  \right\|_{ L^2(\P; H) }
\\&\quad+ 
  \adjustlimits\limsup_{ M \rightarrow \infty }
  \sup_{ t \in [0,T] }
  \left\|
    \int_0^t
    ( P_N e^{(t-s)A} - P_N e^{(t-\fl{s})A} ) \, dW_s
  \right\|_{ L^2(\P; H) }
\\&=
  \sup_{ t \in [0,T] }
  \left\|
    \int_0^t
    ( e^{(t-s)A} - P_N e^{(t-s)A} ) \, dW_s
  \right\|_{ L^2(\P; H) }
  =
  \sup_{ t \in [0,T] }
  \| O_t - P_N O_t \|_{ \L^2(\P; H) } .
\end{split}
\end{align}
Furthermore, observe that 
the triangle inequality,
Lemma~\ref{lem:PN_eA_convergence},
and~\eqref{eq:lower_bound_4}
ensure that 
\begin{align}
\begin{split}
 &\liminf_{ M \rightarrow \infty }
  \| O_T - \OMNN{T} \|_{ \L^2(\P; H) }
\\&=
  \liminf_{ M \rightarrow \infty }
  \left\|
    \int_0^T
    ( e^{(T-s)A} - P_N e^{(T-\fl{s})A} ) \, dW_s
  \right\|_{ L^2(\P; H) }
\\&=
  \liminf_{ M \rightarrow \infty }
  \bigg\|\!
    \int_0^T
    ( e^{(T-s)A} - P_N e^{(T-s)A} ) \, dW_s
\\&\quad+
    \int_0^T
    ( P_N e^{(T-s)A} - P_N e^{(T-\fl{s})A} ) \, dW_s
  \bigg\|_{ L^2(\P; H) }
\\&\geq
  \liminf_{ M \rightarrow \infty }
  \left\|
    \int_0^T
    ( e^{(T-s)A} - P_N e^{(T-s)A} ) \, dW_s
  \right\|_{ L^2(\P; H) }
\\&\quad-
  \liminf_{ M \rightarrow \infty }
  \left\|
    \int_0^T
    ( P_N e^{(T-s)A} - P_N e^{(T-\fl{s})A} ) \, dW_s
  \right\|_{ L^2(\P; H) }
\\&=
  \left\|
    \int_0^T
    ( e^{(T-s)A} - P_N e^{(T-s)A} ) \, dW_s
  \right\|_{ L^2(\P; H) }
  =
  \| O_T - P_N O_T \|_{ \L^2(\P; H) }
\\&=
  \sup_{ t \in [0,T] }
  \| O_t - P_N O_t \|_{ \L^2(\P; H) } .
\end{split}
\end{align}
Combining this with~\eqref{eq:lower_bound_4}--\eqref{eq:lower_bound_6}
completes the proof of Lemma~\ref{lem:lower_bound_2}.
\end{proof}

\subsection{Lower and upper bounds for 
strong approximation errors of full discretizations
of linear stochastic heat equations}

\begin{theorem}
\label{thm:lower_bounds}
Assume the setting in Section~\ref{sec:lower_bounds_setting}
and let $ M, N \in \N $. Then
\begin{align}
\begin{split}
 &\frac{1}{M^{\nicefrac{1}{4}}}
  \left[ 
    \int_0^{\max\left\{0, \frac{T(N+1)^2}{2M}-\left[1+\frac{\sqrt{T}}{\sqrt{2M}}\right]^2\right\}} 
    \frac{
      \sqrt{T}
      \left[ 1 - e^{-\nu \pi^2 T} \right]
      \left[ 1 - \exp(-\nu\pi^2 \min\{1,\frac{TN^2}{2M}\}) \right]^2
    }{
      32 \nu \pi^2 \sqrt{2} 
      (x+[1+\sqrt{T}]^2)^{\nicefrac{3}{2}} 
    } \, dx
  \right]^{\nicefrac{1}{2}}
\\&\quad+
  \frac{ 1 }{ N^{\nicefrac{1}{2}} }
  \left[
    \frac{ \sqrt{ 1 - e^{-\nu T} } }{ 4 \pi \sqrt{\nu} }  
  \right]
\\&\leq
  \|
    O_T
    -
    \OMNN{T}
  \|_{ \L^2(\P; H) }
  \leq 
  \sup_{ t \in [0,T] }
  \|
    O_t
    -
    \OMNN{t}
  \|_{ \L^2(\P; H) }
\\&\leq  
  \frac{1}{M^{\nicefrac{1}{4}}}
  \left[
    \frac{\sqrt{ T }}{ 2 }
    \left( 
      \frac{ 
	1
      }{
	\pi \sqrt{\nu}
      }
      +
      \frac{ 
	1 
      }{
	\nu \pi^2
      }
      +
      4
      \pi
      \sqrt{\nu}
    \right)
  \right]^{\nicefrac{1}{2}}
  +
  \frac{ 1 }{ N^{\nicefrac{1}{2}} }
  \left[
    \frac{ 1 }{ \pi \sqrt{ 2 \nu } }  
  \right] .
\end{split}
\end{align}
\end{theorem}
\begin{proof}[Proof of Theorem~\ref{thm:lower_bounds}]
Observe that the fact that $ P_N $ is self-adjoint
ensures for all 
$ x \in H $, $ y \in P_N(H) $ that
\begin{align}
\begin{split}
&
  \langle
    x - P_N(x),
    P_N(x) - y
  \rangle_H
\\
&=
  \langle
    x - P_N(x),
    P_N(x) - P_N(y)
  \rangle_H
  =
  \langle
    x - P_N(x),
    P_N(x - y)
  \rangle_H
\\&=
  \langle
    P_N(x - P_N(x)),
    x - y
  \rangle_H
  =
  \langle
    P_N(x) - P_N(x),
    x - y
  \rangle_H
\\&=
  \langle
    0,
    x - y
  \rangle_H
  =
  0 .
\end{split}
\end{align}
This implies for all $ t \in [0,T] $ that
\begin{align}
\begin{split}
 &\|
    O_t
    -
    \OMNN{t}
  \|_{ \L^2(\P; H) }^2
\\&=
  \E\!\left[ 
    \| 
      O_t 
      - 
      \OMNN{t}
    \|_H^2
  \right] 
  =
  \E\!\left[
    \| 
      O_t 
      -
      P_N O_t
      +
      P_N O_t
      - 
      \OMNN{t}
    \|_H^2
  \right] 
\\&=
  \E\!\left[
    \| 
      O_t 
      -
      P_N O_t
    \|_H^2
  \right]
  +
  2 \,
  \E\!\left[
    \langle 
      O_t 
      -
      P_N O_t,
      P_N O_t
      - 
      \OMNN{t}
    \rangle_H
  \right]
\\&\quad+
  \E\!\left[
    \| 
      P_N O_t
      - 
      \OMNN{t}
    \|_H^2
  \right]
\\&=
  \|
    O_t 
    -
    P_N O_t
  \|_{ \L^2(\P; H) }^2
  +
  \|
    P_N O_t
    -
    \OMNN{t}
  \|_{ \L^2(\P; H) }^2 .
\end{split}
\end{align}
Combining this with Lemma~\ref{lem:lower_bound_1}, Lemma~\ref{lem:lower_bound_2},
and the fact that 
\begin{equation}
  \forall \, x,y \in [0,\infty)
  \colon
  \nicefrac{\sqrt{x}}{2}
  +
  \nicefrac{\sqrt{y}}{2}
  \leq 
  \max\{\sqrt{x},\sqrt{y}\}
  \leq
  \sqrt{x+y}
  \leq 
  \sqrt{x} + \sqrt{y}
\end{equation}
completes the proof of Theorem~\ref{thm:lower_bounds}.
\end{proof}

\def\cprime{$'$} \def\cprime{$'$} \def\cprime{$'$} \def\cprime{$'$}
  \def\polhk#1{\setbox0=\hbox{#1}{\ooalign{\hidewidth
  \lower1.5ex\hbox{`}\hidewidth\crcr\unhbox0}}}

\end{document}